\documentclass[reqno,10pt]{amsart}

\usepackage{amsmath, amsthm, amsfonts}
\usepackage{pdfsync}
\usepackage{parskip}
\usepackage{hyperref}
\usepackage{mathrsfs}  
\usepackage{mathtools}
\usepackage{mdframed}
\usepackage{tikz}
\numberwithin{equation}{section}
\newtheorem{theorem}{Theorem}[section]
\newtheorem{prop}[theorem]{Proposition}
\newtheorem{ppt}[theorem]{Property}
\newtheorem{definition}[theorem]{Definition}
\newtheorem{lem}[theorem]{Lemma}

\theoremstyle{remark}
\newtheorem{remark}[theorem]{Remark}

\def\M{\mathsf{M}}
\def\sM{\mathscr{M}}
\def\p{\mathsf{p}}
\def\d{{\sf d}}

\def\ev{\mathsf{C}}

\def\J{\mathbb{J}}

\def\R{\mathbb{R}}
\def\C{\mathbb{C}}

\def\N{\mathbb{N}}
\def\Nz{\mathbb{N}_0}
\def\A{\mathfrak{A}}
\def\K{\frak{K}}

\def\bE{\mathbb{E}}

\def\L{\mathcal{L}}
\def\Lis{\mathcal{L}{\rm{is}}}
\def\F{\mathfrak{F}}

\def\B{\mathbb{B}}

\def\Ok{\mathsf{O}_{\kappa}}

\def\vpk{\varphi_{\kappa}}
\def\psk{\psi_{\kappa}}

\def\Q{\mathsf{Q}^{m}}

\def\cA{\mathcal{A}}
\def\sA{\mathscr{A}}

\def\a{\mathfrak{a}}

\def\kf{\psi^{\ast}_{\kappa}}

\def\Rp{{\rm{Re}}}
\def\id{{\rm{id}}}

\def\gd{{\rm{grad}}}

\def\div{{\rm div}}

\def\sg{{\rm{sign}}}

\def\S{\mathcal{S}}

\def\zW{\prescript{}{0}{W}}

\begin{document}

\title[Global solutions to PME on singular manifolds]{Global solutions to the Porous medium equations on singular manifolds}

\author[Y. Shao]{Yuanzhen Shao}
\address{Department of Mathematics,
         Purdue University, 
         150 N. University Street, 
         West Lafayette, IN 47907, USA}
\email{shao92@purdue.edu}


\begin{abstract}
The main objective of the present work is to discuss the global existence and stability of solutions to the porous medium equations on Riemannian manifolds with singularities. Several different types of solutions are considered. Our proof is based on a Poincar\'e inequality on function space defined on manifolds with singularities.
\end{abstract}
\maketitle

\section{\bf Introduction}

In this paper, we will discuss the global solutions to the porous medium equation on a Riemannian manifold $(\M,g)$ with singularities
\begin{equation}
\label{PME}
\left\{\begin{aligned}
\partial_t u -\Delta_g (|u|^{n-1}u  )&=0   &&\text{on}&&\M\times (0,\infty);\\
u(0)&=u_0   &&\text{on}&&\M
\end{aligned}\right.
\end{equation}
for $n>1$.
Some authors use the convention of $|u|^{n-1}u$ being replaced by $u^n$ for only non-negative solutions. 
Here $\Delta_g$ denotes the Laplace-Beltrami operator with respect to the metric $g$.
When  $n\in (0,1)$, this equation is called the fast diffusion equation.

The porous medium equation is probably one of the most famous nonlinear degenerate parabolic equations, which has been extensively studied by many mathematicians.
J.L.~V\'azquez \cite{Vaz92, Vaz07} proved existence and uniqueness of non-negative weak solutions of Dirichlet problems for the porous medium equation. In a landmark article \cite{DasHam98}, P.~Daskalopoulos and R.~Hamilton showed existence and uniqueness of smooth solutions for the porous medium equation, and the smoothness of the free boundary, namely, the boundary of the support of the solution, under mild assumptions on the initial data. 

In the past decade, there has been rising interest in investigating the porous medium equation on Riemannian manifolds. 
In the compact case, 
we would like to refer the reader to \cite[Section~11.5]{Vaz07} for a survey of existence and uniqueness of different types of solutions to \eqref{PME}.
In another paper \cite{OttoWest05} by F.~Otto and M.~Westdickenberg, based on the gradient flow structure of the porous medium equation proposed by the first author in \cite{Otto01}, 
they proved the contraction 
of solutions to \eqref{PME}  
in the Wasserstein distance on compact manifolds.

The analysis of the porous medium equation on non-compact manifolds heavily relies on the curvature conditions of the manifold under consideration.  For example, even in the linear case of the porous medium equation, namely the heat equation, in a seminal paper by P. Li and S.T. Yau, they showed that the following Li-Yau gradient estimate 
$$
\frac{|\nabla u|^2}{u^2} - \alpha\frac{\partial_t u}{u} \leq \frac{m \alpha^2 K}{2(\alpha-1)} +\frac{m\alpha^2}{2t}
$$
holds on an $m$-dimensional complete manifold with Ricci curvature ${\rm Ric}\geq -K$ for some $K\geq 0$, where $u$ is a positive smooth solution to the heat equation and $\alpha>1$.
This suggests that the curvature conditions play an important role in the analysis on non-compact manifolds.
S.T.~Yau \cite{Yau94} proved a similar type of gradient estimate  for \eqref{PME}, but this estimate also depends on the derivatives of the initial data.
An analogue to the Li-Yau estimate for the general porous medium equation is obtained by D.G. Aronson and  P. B\'enilan in \cite{AronBen79} on $\R^N$, which reads as follows:
$$
\frac{|\nabla f|^2}{(n-1)f + n} + \frac{\partial_t f}{(n-1)f + n}\leq \frac{K}{2t},
$$
where $\displaystyle f=n\frac{u^{n-1}-1}{n-1}$ with $u$ being a positive solution to \eqref{PME} and 
$\displaystyle \frac{2}{K}=\frac{2}{N}+(n-1)$.
The first Aronson-B\'enilan and Li-Yau type estimate for the porous medium equation is attributed to J.L.~V\'azquez \cite{Vaz07} for positive solutions on complete manifolds 
with non-negative Ricci curvature. 
Later, under the weaker assumption ${\rm Ric}\geq -K$ for some $K\geq 0$, P. Lu, L. Ni, J.L. V\'azquez and C. Villani \cite{LuNiVazVil09} derived more general Aronson-B\'enilan and Li-Yau type gradient estimates for the porous medium equation and some fast diffusion equations. 
Under the same condition on the Ricci curvature,  a similar non-local result was obtained by G. Huang and Z. Huang, H. Li \cite{HuaHuaLi13}, and in the same article the authors also proved a Harnack-type inequality for the porous medium equation. 
Under the same geometric assumption, X. Xu \cite{Xu12} proved a localized Hamilton-type inequality for bounded positive solutions to \eqref{PME}. 
Z. Qian and Z. Zhang \cite{QianZhang14} also derived a similar Aronson-B\'enilan type gradient estimate for equations of the form $\partial_t u-\Delta u^n =V u^n$, where $V$ is a smooth vector field, satisfying the Bakry-mery curvature-dimension condition. 
Interested readers may also refer to \cite{ChaLee12, HuaXuZen16, Zhu13} for related results and some Liouville-type theorems for the porous medium equation.

A well-known phenomenon of the porous medium equation is its smoothing effect. For example, the solution to \eqref{PME} on bounded Euclidean domains with homogeneous Dirichlet boundary
conditions satisfies
$$
\|u(t)\|_\infty \leq \frac{C}{t^{\frac{1}{n-1}}}.
$$
To the best of the author's knowledge, J.L. V\'azquez was the first to study this effect on manifolds. More precisely, in \cite{Vaz14} he established the following inequality for the solution to \eqref{PME} 
on the hyperbolic space $\mathbb{H}^d$
$$
\|u(t)\|_\infty \leq C (\frac{\ln t}{t})^{\frac{1}{n-1}}.
$$
Shortly afterwards,  G.~Grillo and M.~Muratori \cite{GriMur16} generalized this result to a class of complete and simply connected manifolds whose sectional curvature is bounded above by a negative constant. 
More recent achievements in this direction can be found in \cite{GriMurVaz16}. 
In this article, the authors also obtained a radius estimate for the free boundary of the solution to \eqref{PME} with compactly supported initial data. 
The study of the free boundary of positive solutions to the porous medium equation is another interesting topic. 
To keep this introduction at a reasonable length, we will not go into details on this topic. 
Similar results on the ultracontractivity of solutions to \eqref{PME} for small time have been obtained by M.~Bonforte and G.~Grillo in \cite{BonGri05}. 
In another paper \cite{GriMurPun15} by G.~Grillo, M.~Muratori and F.~Punzo, they established existence and uniqueness of solutions to \eqref{PME} with $u_0$ being finite Radon measures under the same geometric assumptions as in \cite{GriMur16}. Finally, interested readers can find more interesting topics related to the porous medium equation on manifolds in \cite{CaoZhu15, Dek08, Dem05, Dem08, KisShtZla08}.

In the case of manifolds with singularities, apparently, we will lose the benefit of the curvature conditions proposed in the work mentioned above. 
To the best of the author's knowledge, the only results concerning  the porous medium equations in the singular manifolds setting are contained in the following two papers \cite{RoiSch15, Shao14}. 

In \cite{RoiSch15}, the authors established the short time existence of solution to \eqref{PME} on manifolds with conical singularities. This result is based on the earlier work of \cite{RoiSch13, RoiSch14, SchSei05} on bounded imaginary powers of the Laplace-Beltrami operator on Mellin-Sobolev spaces. More precisely, given a compact closed manifold $B$, the Laplace-Beltrami on the conical manifold $([0,1)\times B) / (\{0\}\times B)$  reads as 
$$\Delta_c=t^{-2}((t \partial_t)^2 + ({\rm dim} B-1)(t\partial_t) + \Delta_B). $$
Based on the $\mathscr{R}$-sectoriality of $\Delta_c$, in \cite{RoiSch15}, the authors established the $\mathscr{R}$-sectoriality for operators of the form $-\bar{u}\Delta_c$ with $\mathscr{R}$-angle $<\pi/2$ on a conical manifold $([0,1)\times B) / (\{0\}\times B)$, where $\bar{u}=C+ u^*$, by a localization argument and the perturbation results of $\mathscr{R}$-sectorial operators. Here $C>0$ is a constant and  the function $u^*\to 0$ while approaching $\{0\}\times B$ and $u^*$ belongs to some Mellin-Sobolev spaces. Then the local well-posedness of equation~\eqref{PME} follows from the $\mathscr{R}$-sectoriality and maximal $L_p$-regularity theory for initial data of the same form as $\bar{u}$. 
The interested reader can refer to \cite{DenHiePru03, Pru03, PruSim16} for more details of operators of bounded imaginary powers, $\mathscr{R}$-sectorial operators, and maximal regularity theory. 	
This result is comparable to the first part of Section~5 below, where local well-posedness of \eqref{PME} has been established for a class of singular manifolds, including but not restricted to conical manifolds. 

In \cite{Shao14}, we showed that \eqref{PME} admits a local in time solution for initial data degenerating like $\d^\alpha$, where $\d$ can be taken to be the distance to the singular sets and $\alpha>2/(n-1)$. By this choice, the set of degeneracy of solutions to \eqref{PME} in \cite{Shao14} is the singularities of the background manifolds and this solution has a waiting-time phenomenon as was well-known for solutions to the porous medium equation in Euclidean spaces.

In the current paper, our first goal is to extend the local well-posedness result for the porous medium equation obtained in \cite{RoiSch15} to a larger class of singular manifolds, and then we will prove the (exponentially) asymptotic stability for the equilibria of the porous medium equation on singular manifolds. 
These results are established in a weighted Sobolev space framework. 
See Section~2 for more details of these spaces.
The cornerstone for our results is a Poinca\'e inequality \eqref{S4: Poincare ineq} on a class of singular manifolds coupled with a form operator argument, based upon which we establish the maximal $L_p$-regularity property for a class of differential operators of the form 
\begin{equation}
\label{S1: DO}
[u\mapsto -\div_g(a \gd_g u)].
\end{equation}
Here $\div_g$ and $\gd_g$ are the divergence and gradient operators with respect to the singular metric $g$, and $a$ is a $C^1$ function satisfying $\inf a >0$ and $\|a\|_\infty<\infty$. 
So roughly speaking, this result relaxes the asymptotic condition  (toward the singular ends) of the coefficient $\bar{u}$ of the differential operator $-\bar{u}\Delta_c$   imposed in \cite{RoiSch15}.
From the Poinca\'e inequality \eqref{S4: Poincare ineq}, a spectrum bound for \eqref{S1: DO} can be easily derived. 
This bound is enough to obtain the stability of the stationary solutions to the porous medium equations via the linearized stability theory.

Another approach to nonlinear parabolic problems is the implicit time discretization method. To be precise, we first write a quasilinear parabolic equation into an abstract ODE problem of the form
\begin{equation}
\label{S1: ODE}
\frac{d}{dt}u +\cA(u)=0,\quad u(0)=u_0,
\end{equation}
where $\cA$ is a nonlinear operator acting in a Banach space $X$. (We can actually allow $\cA$ to be multi-valued, but for the sake of simplicity, we always assume in this paper that $\cA$ is single-valued, which is enough for the theory of the porous medium equation.) Then we can approximate the solution in the following way: we partition the time interval $J=[0,T]$ into $\mathcal{P}=\{0=t_0\leq t_1\leq \cdots \leq t_{N-1} \leq t_N= T\}$. Put $u_i:=u(t_i)$ and $\delta_i:=t_i-t_{i-1}$. Then we can solve $u_i$ from $u_{i-1} $ implicitly from
$$
\frac{u_i- u_{i-1}}{\delta_i} +\cA(u_i)= 0,\quad i=1,\cdots,N.
$$
If the resolvent operator $R(\lambda,\cA):=(\id+\lambda\cA)^{-1}$ are well defined as operators in $X$ for all $\lambda>0$, we can solve
$$
u_i= R(\delta_i,\cA)u_{i-1} .
$$
Then we can piece together the discrete solution to obtain an approximate solution to \eqref{S1: ODE} by, for example, piecewise constant functions or piecewise affine functions.

To consider the convergence of the approximate solutions, we will need the concept of $\varepsilon$-discretization, by which we mean a partition of $J$ such that
$$
t_i-t_{i-1}<\varepsilon.
$$
A function $u\in C(J,X)$ is called a mild solution to \eqref{S1: ODE} if for every $\varepsilon>0$, there exists an $\varepsilon$-discretization with an approximate solution $u_\varepsilon$ such that
$$
\|u(t)- u_\varepsilon(t)\|_X\leq \varepsilon,\quad t\in J.
$$
The well-known Crandall-Liggett theorem says that the abstract Cauchy problem \eqref{S1: ODE} has a unique mild solution for every $u_0\in \overline{dom(\cA)}$ if the operator $\cA$ is $m$-accretive. The definition of $m$-accretive operators can be found in Section~6. Interested readers may refer to \cite{Bar73,CraLig71,Vaz07} for more details of mild solutions and nonlinear semigroup theory. In Section~6, we show that, on a singular manifold $(\M,g)$, \eqref{PME} admits a global mild solution $u\in C(J,L_1(\M))$ for any interval $J=[0,T]$ and $u_0\in L_1(\M)$.

This paper is organized as follows. In Section~2, we will state the concept of singular manifolds used in this paper and the precise definition of the function spaces on singular manifolds. 
Section~3 is the cornerstone of this paper, in which we will prove a Poincar\'e inequality on a class of singular manifolds, including edge or conic manifolds. Based on this inequality, in Section~4, we will show the crucial contraction properties of the semigroup generated by a class of differential operators including the Laplace-Beltrami operator. These properties give the theoretic basis of the lineaized stability argument in Section~5 and the $m$-accretivity of the nonlinear operator in Section~6. 
In Section~5, we will first generalize the local well-posedness result of the porous medium equation on conic manifolds obtained by N.~Roidos and E.~Schrohe in \cite{RoiSch14} to a larger class of singular manifolds, and then we will derive the stability of equilibria in the framework of the well-posedness theorem. Under the same geometric assumptions as in Section~5, in the last section, we will prove the global existence of mild solutions for the porous medium equation in the spirit of the limit of the implicit time discretization solutions introduced above.

\textbf{Notations:} 
Given any topological set $U$, $\mathring{U}$ denotes the interior of $U$. 
\smallskip\\
For any two Banach spaces $X,Y$, $X\doteq Y$ means that they are equal in the sense of equivalent norms. The notation $\Lis(X,Y)$ stands for the set of all bounded linear isomorphisms from $X$ to $Y$.
\smallskip\\
Given any Banach space $X$ and  manifold $\mathscr{M}$,
let $\| \cdot \|_\infty$, $\| \cdot \|_{s,\infty}$, $\|\cdot \|_p$ and $\|\cdot \|_{s,p}$ denote the usual norm of the $X$-valued Banach spaces $BC(\mathscr{M},X)$(or $L_\infty(\mathscr{M},X)$), $BC^s(\mathscr{M},X)$, $L_p(\mathscr{M},X)$ and $W^s_p(\mathscr{M},X)$, respectively. If the letter $X$ is omitted in the definition of these space, e.g., $BC(\sM)$, it means the corresponding space is $\C$-valued.
\smallskip\\ 

\section{\bf Singular manifolds and weighted function spaces}

\subsection{\bf Uniformly regular Riemannian manifolds and Singular manifolds}

The following concepts of {\em uniformly regular Riemannian manifolds} and {\em singular manifolds} were introduced by H.~Amann in \cite{Ama13, AmaAr}. For the aim of this paper, we will restrict ourselves to only manifolds without boundary. For boundary value problems, the reader may refer to \cite{Ama13b, Shao1603}. The result in this paper can be generalized to {\em singular manifolds} with boundary by using the results in \cite{Shao1603}.

Let $(\M,g)$ be a $C^\infty$-Riemannian manifold of dimension $m$   without boundary endowed with $g$ as its Riemannian metric such that its underlying topological space is separable. An atlas $\A:=(\Ok,\vpk)_{\kappa\in \K}$ for $\M$ is said to be normalized if 
$
\vpk(\Ok)=\Q,
$
where  $\Q$ is the unit cube at the origin in $\R^m$. We put $\psk:=\vpk^{-1}$. 

The atlas $\A$ is said to have \emph{finite multiplicity} if there exists $K\in \N $ such that any intersection of more than $K$ coordinate patches is empty. Put
\begin{align*}
\mathfrak{N}(\kappa):=\{\tilde{\kappa}\in\K:\mathsf{O}_{\tilde{\kappa}}\cap\Ok\neq\emptyset \}.
\end{align*} 
The finite multiplicity of $\A$ and the separability of $\M$ imply that $\A$ is countable.

An atlas $\A$ is said to fulfil the \emph{uniformly shrinkable} condition, if it is normalized and there exists $r\in (0,1)$ such that $\{\psk(r{\Q}):\kappa\in\K\}$ is a cover for ${\M}$. 

$(\M,g)$ is said to be a {\em{uniformly regular Riemannian manifold}} if it admits an atlas $\A$ such that
\begin{itemize}
\item[(R1)] $\A$ is uniformly shrinkable and has finite multiplicity. If $\M$ is oriented, then $\A$ is orientation preserving.
\item[(R2)] $\|\varphi_{\eta}\circ\psk \|_{k,\infty}\leq c(k) $, $\kappa\in\K$, $\eta\in\mathfrak{N}(\kappa)$, and $k\in{\N}_0$.
\item[(R3)] $\kf g\sim g_m $, $\kappa\in\K$. Here $g_m$ denotes the Euclidean metric on ${\R}^m$ and $\kf g$ denotes the pull-back metric of $g$ by $\psk$.
\item[(R4)] $\|\kf g\|_{k,\infty}\leq c(k)$, $\kappa\in\K$ and $k\in\Nz$.
\end{itemize}
Here $\|u\|_{k,\infty}:=\max_{|\alpha|\leq k}\|\partial^{\alpha}u\|_{\infty}$, and it is understood that a constant $c(k)$, like in (R2), depends only on $k$. An atlas $\A$ satisfying (R1) and (R2) is called a \emph{uniformly regular atlas}. (R3) reads as
\begin{center}
$|\xi|^2/c\leq \kf g(x)(\xi,\xi) \leq{c|\xi|^2}$,\hspace{.5em} for any $x\in \Q,\xi\in \R^m, \kappa\in\K$ and some $c\geq{1}$.
\end{center}

Assume that $\rho\in C^{\infty}(\M,(0,\infty))$. Then $(\rho,\K)$ is a {\em singularity datum} for $\M$ if
\begin{itemize}
\item[(S1)] $(\M,g/\rho^2)$ is a {\em uniformly regular Riemannian manifold}.
\item[(S2)] $\A$ is a uniformly regular atlas.
\item[(S3)] $\|\kf\rho\|_{k,\infty}\leq c(k)\rho_{\kappa}$, $\kappa\in\K$ and $k\in\N_0$, where $\rho_{\kappa}:=\rho(\psk(0))$.
\item[(S4)] $\rho_{\kappa}/c\leq \rho(\p)\leq c\rho_{\kappa}$, $\p\in\Ok$ and $\kappa\in\K$ for some $c\geq 1$ independent of $\kappa$.
\end{itemize}
Two {\em singularity data} $(\rho,\K)$ and $(\tilde{\rho},\tilde{\K})$ are equivalent, if
\begin{itemize}
\item[(E1)] $\rho\sim \tilde{\rho}$.
\item[(E2)] card$\{\tilde{\kappa}\in\tilde{\K}:\mathsf{O}_{\tilde{\kappa}}\cap\Ok\neq\emptyset\}\leq c$, $\kappa\in\K$.
\item[(E3)] $\|\varphi_{\tilde{\kappa}}\circ\psk\|_{k,\infty}\leq{c(k)}$, $\kappa\in\K$, $\tilde{\kappa}\in\tilde{\K}$ and $k\in{\N}_0$
\end{itemize}
We write the equivalence relationship as $(\rho,\K)\sim(\tilde{\rho},\tilde{\K})$. (S1) and (E1) imply that 
\begin{align*}
1/c\leq \rho_{\kappa}/\tilde{\rho}_{\tilde{\kappa}}\leq c,\hspace*{.5em} \kappa\in\K,\hspace*{.5em} \tilde{\kappa}\in\tilde{\K}\text{ and }\mathsf{O}_{\tilde{\kappa}}\cap\Ok\neq \emptyset.
\end{align*}
{\em A singularity structure}, $\mathfrak{S}(\M)$, for $\M$ is a maximal family of equivalent {\em singularity data}. A {\em singularity function} for $\mathfrak{S}(\M)$ is a function $\rho\in C^{\infty}(\M,(0,\infty))$ such that there exists an atlas $\A$ with $(\rho,\A)\in\mathfrak{S}(\M)$. The set of all {\em singularity functions} for $\mathfrak{S}(\M)$ is the {\em singular type}, $\mathfrak{T}(\M)$, for $\mathfrak{S}(\M)$. By a {\em{singular manifold}} we mean a Riemannian manifold $\M$ endowed with a singularity structure $\mathfrak{S}(\M)$. Then $\M$ is said to be \emph{singular of type} $\mathfrak{T}(\M)$. If $\rho\in\mathfrak{T}(\M)$, then it is convenient to set $[\![\rho]\!]:=\mathfrak{T}(\M)$ and to say that $(\M,g;\rho)$ is a {\em singular manifold}. A {\em singular manifold} is a {\em uniformly regular Riemannian manifold} iff $\rho\sim {\bf 1}_{\M}$. 



Lastly, for each $k\in\N$, the concept of {\em{$C^k$-uniformly regular Riemannian manifold}} is defined by modifying (R2), (R4) and (L1), (L2) in an obvious way. Similarly, {\em{$C^k$-singular manifolds}} are defined by replacing the smoothness of $\rho$ by $\rho\in C^k(\M,(0,\infty))$ and altering (S1)-(S3) accordingly.

\subsection{\bf Tensor bundles}
Suppose $(\M,g;\rho)$ is a {\em singular manifold}.
Given $\sigma,\tau\in\N_0$, 
$$T^{\sigma}_{\tau}{\M}:=T{\M}^{\otimes{\sigma}}\otimes{T^{\ast}{\M}^{\otimes{\tau}}}$$ 
is the $(\sigma,\tau)$-tensor bundle of $\M$, where $T{\M}$ and $T^{\ast}{\M}$ are the (complexified) tangent and the cotangent bundle of ${\M}$, respectively.
We write $\mathcal{T}^{\sigma}_{\tau}{\M}$ for the $C^{\infty}({\M})$-module of all smooth sections of $T^{\sigma}_{\tau}\M$,
and $\Gamma(\M,T^{\sigma}_{\tau}{\M})$ for the set of all sections.

We denote by $\nabla=\nabla_g$ the extension of the Levi-Civita connection over $\mathcal{T}^{\sigma}_{\tau}{\M}$.
Set $\nabla_{i}:=\nabla_{\partial_{i}}$ with $\partial_{i}=\frac{\partial}{\partial{x^i}}$. 
For $k\in\Nz$, we define
$$\nabla^k: \mathcal{T}^{\sigma}_{\tau}{\M}\rightarrow{\mathcal{T}^{\sigma}_{\tau+k}{\M}},\quad a\mapsto{\nabla^k{a}}$$
by letting $\nabla^0 a:=a$ and $\nabla^{k+1} a:=\nabla\circ\nabla^k a$.
We can also extend the Riemannian metric $(\cdot|\cdot)_g$ from the tangent bundle to any $(\sigma,\tau)$-tensor bundle $T^{\sigma}_{\tau}{\M}$, which is still written as $(\cdot|\cdot)_g$. Meanwhile, $(\cdot|\cdot)_{g^*}$ stands for the induced contravariant metric. 
In addition,
$$|\cdot|_g:=|\cdot|_{g^\tau_\sigma}:\mathcal{T}^{\sigma}_{\tau}{\M}\rightarrow{C^{\infty}}({\M}),\quad a\mapsto\sqrt{(a|a)_g}$$
is called the (vector bundle) \emph{norm} induced by $g$.
\smallskip\\
We assume that $V$ is a $\C$-valued tensor bundle on $\M$, i.e.,
$$V=V^{\sigma}_{\tau}:=\{T^{\sigma}_{\tau}\M, (\cdot|\cdot)_g\},$$ 
for some $\sigma,\tau\in\N_0$. 

Throughout the rest of this paper, unless stated otherwise, we always assume that 
\smallskip
\begin{mdframed}
\begin{itemize}
\item $(\M,g;\rho)$ is a {\em singular manifold} with $\rho\leq 1$.
\item $\rho\in \mathfrak{T}(\M)$, $s\geq 0$, $k\in\Nz$, $1<p<\infty$ and $\vartheta\in \R$.
\item $\sigma,\tau\in \Nz$, $V=V^{\sigma}_{\tau}:=\{T^{\sigma}_{\tau}\M, (\cdot|\cdot)_g\}$.
\end{itemize}
\end{mdframed}


\subsection{\bf Weighted function spaces}

We denote by $\mathcal{D}({\M},V)$ the space of smooth sections of $V$ that is compactly supported in $\M$.
Then the weighted Sobolev space $W^{k,\vartheta}_p({\M},V)$ is defined as the completion of $\mathcal{D}({\M},V)$ in $L_{1,loc}(\M,V)$ with respect to the norm
\begin{center}
$\|\cdot\|_{k,p;\vartheta}: u\mapsto(\sum_{i=0}^{k}\|\rho^{\vartheta+i+\tau-\sigma}|\nabla^{i}u|_{g}\|_p^p)^{\frac{1}{p}}$.
\end{center}
Note that $W^{0,\vartheta}_p(\M, V)=L^{\vartheta}_p(\M, V)$ with equal norms. In particular, we can define the weighted spaces $L_q^\vartheta(\M,V)$ for $q\in \{1,\infty\}$ in a similar manner.

Define
$$BC^{k,\vartheta}(\M,V):=(\{u\in{C^k({\M},V)}:\|u\|_{k,\infty;\vartheta}<\infty\},\|\cdot\|_{k,\infty;\vartheta}),$$
where $\|u\|_{k,\infty;\vartheta}:={\max}_{0\leq{i}\leq{k}}\|\rho^{\vartheta+i+\tau-\sigma}|\nabla^{i}u|_{g}\|_{\infty}$.
We also set
$$BC^{\infty,\vartheta}({\M},V):=\bigcap_{k}BC^{k,\vartheta}(\M, V).$$ 
The weighted Sobolev-Slobodeckii spaces are defined as
\begin{align*}
W^{s,\vartheta}_p({\M},V):=(L^{\vartheta}_p(\M, V),W^{k,\vartheta}_p(\M, V))_{s/k,p},
\end{align*}
for $ s\in \R_+\setminus\Nz$, $k=[s]+1$,
where $(\cdot,\cdot)_{\theta,p}$ is the real interpolation method \cite[Chapter I.2.4.1]{Ama95}.
\smallskip\\

We denote by $\ev^{\sigma+1}_{\tau+1}: V^{\sigma+1}_{\tau+1}\rightarrow V^\sigma_\tau$ the contraction with respect to position $\sigma+1$ and $\tau+1$, that is for any $(i)\in\J^\sigma$, $(j)\in\J^\tau$ and $k,l\in\J^1$ and $\p\in\M$
$$
\ev^{\sigma+1}_{\tau+1}a  :=\ev^{\sigma+1}_{\tau+1} a^{(i;k)}_{(j;l)} \frac{\partial}{\partial x^{(i)}}\otimes \frac{\partial}{\partial x^k}\otimes dx^{(j)}\otimes dx^l:=a^{(i;k)}_{(j;k)} \frac{\partial}{\partial x^{(i)}} \otimes dx^{(j)} 
$$
in every local chart. Recall that the surface divergence of tensor fields with respect to the metric $g$ is the map
\begin{equation*}
\div=\div_g: C^1(\M,V^{\sigma+1}_\tau)\rightarrow C(\M, V^\sigma_\tau), \quad a\mapsto \ev^{\sigma+1}_{\tau+1}(\nabla a).
\end{equation*}



In the rest of this subsection, we will present some properties of weighted function spaces. Most of their proofs can be found in \cite{Ama13, AmaAr, Shao15}, and will thus be omitted.

\begin{ppt}
\label{S2: pointwise multiplication}
Let $k\in\Nz$. Then $[(v_1,v_2)\mapsto v_1 v_2]$ is a bilinear and continuous map for $k\in\Nz$ and $s\leq k$
$$BC^{k,\vartheta_1}(\M)\times W^{s,\vartheta_2}_p(\M, V)\rightarrow W^{s,\vartheta_1+\vartheta_2}_p(\M, V).$$ 
\end{ppt}

Put $\gd:=\gd_g$ as the gradient with respect to $g$.
\begin{ppt}
\label{S2: nabla}
For $\F\in \{BC, W_p\}$, we have
$$\nabla\in \L(\F^{s+1,\vartheta}(\M, V^\sigma_\tau), \F^{s,\vartheta}(\M, V^\sigma_{\tau+1}),$$
and
$$\gd \in \L(\F^{s+1,\vartheta}(\M, V^\sigma_\tau), \F^{s,\vartheta+2}(\M, V^{\sigma+1}_\tau)),$$
and
$$\div\in\L(\F^{s+1,\vartheta}(\M, V^{\sigma+1}_\tau), \F^{s,\vartheta}(\M, V^\sigma_\tau)).$$
\end{ppt}

\begin{ppt}
\label{S2: change of wgt}
For $\F\in\{BC, W_p\}$, we have
$$f_{\vartheta}:=[u\mapsto \rho^{\vartheta}u] \in \Lis(\F^{s,\vartheta^\prime+\vartheta}(\M, V),\F^{s,\vartheta^\prime}(\M, V)).$$ 
\end{ppt}


\begin{ppt}
\label{S2: interpolation}
Suppose that $k_i\in \Nz$, $\vartheta_i\in \R$ with $i=0,1$, $0<\theta<1$ and $k_0<k_1$. Then
$$(W^{k_0,\vartheta_0}_p(\M,V), W^{k_1,\vartheta_1}_p(\M,V))_{\theta,p} \doteq W^{k_\theta,\vartheta_\theta}_p(\M,V).$$
Here $\xi_\theta:=(1-\theta)\xi_0+\theta \xi_1$ for any $\xi_0,\xi_1\in \R$, and $k_\theta\notin \N$.
\end{ppt}

\begin{ppt}
\label{S2: Sobolev embedding}
Suppose that $\vartheta\in\R$ and $s>k+\frac{m}{p}$. Then
$$W^{s,\vartheta}_p(\M) \hookrightarrow BC^{k,\vartheta+\frac{m}{p}}(\M). $$
\end{ppt}

\begin{ppt}
\label{S2: embedding}
For $\vartheta_1<\vartheta_0$ and $\F\in \{BC, W_p\}$,
$\F^{s,\vartheta_1}(\M) \hookrightarrow \F^{s,\vartheta_0}(\M)$.
\end{ppt}

\begin{ppt}
\label{S2: substituion op}
Let $U^{k,\vartheta}_b=\{u\in BC^{k,\vartheta}(\M):\, \inf \rho^\vartheta u>b\}$ with some $b>0$. Suppose that $\alpha\in\R$. Then 
$$ [u\mapsto u^\alpha]\in C^\omega(U^{k,\vartheta}_b, BC^{k,\vartheta\alpha}(\M)).$$
\end{ppt}
\begin{proof}
By Property~\ref{S2: change of wgt}, $[u\mapsto \rho^\vartheta u]\in C^\omega(U^{k,\vartheta}_b, \hat{U}^k_b)$, where
$\hat{U}^{k,\vartheta}_b: = \{u\in BC^{k,0}(\M):\, \inf u>b\}$. Note that $(\hat{\M},\hat{g}):=(\M,g/\rho^2)$ is a {\em uniformly regular Riemannian manifold}. It is shown in \cite[Section~4]{Ama13b} that
$$BC^{k,0}(\M)\doteq BC^k(\hat{\M}). $$
We can slightly modify the proof of \cite[Proposition~6.3]{ShaoSim13} and conclude that 
$$[u\mapsto u^\alpha]\in C^\omega(\hat{U}^{k,\vartheta}_b, BC^k(\hat{\M}) ).$$
In view of $u^\alpha= \rho^{-\vartheta\alpha}(\rho^\vartheta u)^\alpha$ and Property~\ref{S2: change of wgt}, the claimed result follows.
\end{proof}

\subsection{\bf Unweighted Sobolev spaces}

The unweighted Sobolev spaces are defined in the usual manner. For simplicity, we only state the least necessary definitions and properties.

Define $W^1_2(\M,V):=$ the closure of $\mathcal{D}(\M,V)$ with respect to the norm $\|\cdot\|_{1,2}$, where
$$
\|u\|_{1,2}:= (\|u\|_2^2 + \||\nabla u|_g\|_2^2)^{1/2}.
$$
Here $\|\cdot\|_2$ is the norm of $L_2(\M,V)$. 

For any $u,v\in \mathcal{D}(\M,V)$, we define the duality pairing $\langle \cdot | \cdot \rangle$ of $L_2(\M)$ as follows.
$$\langle u | v\rangle:= \int\limits_\M ( u |\bar{v} )_g\, d\mu_g. $$ 
Then the negative order Sobolev space is defined by
$$
W^{-1}_2(\M,V):=(W^1_2(\M, V))^\prime
$$
with respect to $\langle \cdot | \cdot \rangle$. For any $u\in L_2(\M,V^\sigma_\tau)$, the gradient and divergence are defined as follows. For $v\in \mathcal{D}(\M,V^{\sigma+1}_\tau )$,
\begin{equation}
\label{S2: div-thm-unweighted-1}
\langle \gd u | v\rangle := -\langle u | \div v\rangle,
\end{equation}
and for $v\in \mathcal{D}(\M, V^{\sigma-1}_\tau)$
\begin{equation}
\label{S2: div-thm-unweighted-2}
\langle \div u | v\rangle := -\langle u | \gd v\rangle.
\end{equation}

By a density argument, it is obvious that \eqref{S2: div-thm-unweighted-1} and \eqref{S2: div-thm-unweighted-2} also hold true for $v\in W^1_2(\M,V^{\sigma+1}_\tau )$ or $v\in  W^1_2(\M, V^{\sigma-1}_\tau)$, respectively.
\begin{prop}
\label{S2: nabla-unweighted}
It holds for $k=1,0$ that
$$
\gd\in \L(W^k_2(\M, V^\sigma_\tau), W^{k-1}_2(\M, V^{\sigma+1}_\tau)
,\quad
\div\in\L(W^k_2(\M, V^{\sigma+1}_\tau), W^{k-1}_2(\M, V^\sigma_\tau)).
$$
It is understood that $W^0_2(\M):=L_2(\M)$.
\end{prop}




\section{\bf Poincar\'e inequality on singular manifolds}

In this section, we will prove a Poincar\'e inequality for {\em singular manifolds} with wedge ends and some kind of removable singularities, which is of particular importance in obtaining further stability estimates and maximal regularity results for a class of differential operators on {\em singular manifolds}.

Let $J:=(0,1]$.
Following \cite{Ama14}, we denote by $\mathscr{C}(J)$ the set of all $R\in C^\infty(J,(0,\infty))$ with 
\begin{align*}
\begin{cases}
\text{(i)} \quad & R(1)=1 \text{ and }R(0):=\lim\limits_{t\to 0} R(t)=0;\\
\text{(ii)}  & \int_J dt/R(t)=\infty;\\
\text{(iii)}  & 	\|\partial^k_t R\|_{\infty} <\infty,\quad k\geq 1.
\end{cases}
\end{align*}
The elements in $\mathscr{C}(J)$ are called {\em cusp characteristics} on $J$.
If $R$ further satisfies
$$
\text{(iv)} 
\dot{R}\sim {\bf 1}_J,
\quad
|\ddot{R}| <\infty,
$$
then we call it a {\em uniformly mild cusp characteristic}.
We write $R\in \mathscr{C_U}(J)$.

We will quote several lemmas from \cite{Ama14}, which serve as the cornerstones of the construction of {\em singular manifolds} of wedge ends.
\begin{lem}
\label{S4: lem 4.1}
{\cite[Theorem~3.1]{Ama14}}
Suppose that $\rho$ is a bounded singularity function on $(\M,g)$, and $\tilde{\rho}$ is one for $(\tilde{\M},\tilde{g})$. Then $\rho\otimes\tilde{\rho}$ is a singularity function for $(\M\times\tilde{\M}, g+\tilde{g})$.
\end{lem}

\begin{lem}
\label{S4: lem 4.2}
{\cite[Lemma~3.4]{Ama14}}
Let $f: \tilde{\M}\rightarrow\M$ be a diffeomorphism of manifolds. Suppose that $(\M,g;\rho)$ is a singular manifold. Then so is
$(\tilde{\M},f^*g;f^*\rho)$.
\end{lem}

\begin{lem}
\label{S4: lem 4.3}
{\cite[Lemma~5.2]{Ama14}}
Suppose that $R\in \mathscr{C}(J)$. Then $R$ is a singularity function for $(J, dt^2)$.
\end{lem}

Assume that $(B,g_B)$ is a $d$-dimensional compact closed submanifold of $\R^{\bar{d}}$, and $R\in \mathscr{C}(J)$. We define the (model) $(R,B)$-cusp $P(R,B)$ on $J$, also called $R$-cusp over $B$ on $J$, by 
$$P(R,B)=P(R,B;J):=\{(t,R(t)y):\, t\in J, \, y\in B\}\subset \R^{1+\bar{d}} .$$
It is a $(1+d)$-dimensional submanifold of $\R^{1+\bar{d}}$. The map
$$\phi_P=\phi_P(R):P \rightarrow J\times B: \quad (t,R(t)y)\rightarrow (t,y) $$
is a diffeomorphism, the {\em canonical stretching diffeomorphism} of $P$.

Based on the above three lemmas, we can easily show that
\begin{lem}
$(P(R,B), \phi_P^*(dt^2 + g_B); \phi_P^*(R\otimes {\bf 1}_B))$ is a singular manifold.
\end{lem}

Assume that $(\Gamma, g_\Gamma)$ is a compact connected Riemannian manifold without boundary. Then the (model) $\Gamma$-wedge over the $(R,B)$-cusp, $P(R,B)$, is defined by
$$W=W(R,B,\Gamma):=P(R,B)\times\Gamma.$$
If $\Gamma$ is a one-point space, then $W$ can be naturally identified with $P$. Thus every cusp is also a wedge.

The following result is an immediate conclusion from Lemmas~\ref{S4: lem 4.1}-\ref{S4: lem 4.3}.
\begin{lem}
\label{S4: wedge}
$(W(R,B,\Gamma), \phi_P^*(dt^2 + g_B)+g_\Gamma; \phi_P^*(R\otimes {\bf 1}_B)\otimes {\bf 1}_\Gamma)$ is a singular manifold.
\end{lem}

Suppose that $f: \M\to W(R,B,\Gamma)$. Then $(\M, f^*(\phi_P^*(dt^2 + g_B)+g_\Gamma )  ; f^*(\phi_P^*(R\otimes {\bf 1}_B)\otimes {\bf 1}_\Gamma))$ is called a $\Gamma$-wedge over the $(R,B)$-cusp, and by Lemma~\ref{S4: lem 4.2} is a {\em singular manifold}.

Given any compact submanifold $\Sigma\subset (\M,g)$, the distance function 
is a well-defined smooth function in a collar neighborhood $\mathscr{U}_\Sigma$ of $\Sigma$ (excluding $\Sigma$ itself). The distance ball at $\Sigma$ with radius $r$ is defined by 
$$\B_\M(\Sigma,r):= \{\p\in \M: {\rm dist}_\M(\p,\Sigma)<r \}. $$
\begin{definition}
\label{S4: Torn mfd}
\begin{itemize}
\item[]
\item[(i)] Suppose that $(\sM,g)$ is an $m$-dimensional uniformly regular Riemannian manifold, and $\boldsymbol{\Sigma}=\{\Sigma_j:j=1,\cdots,k\}$ is a finite set of disjoint compact connected submanifolds of codimension $1$  such that $\Sigma_j\subset \mathring{\sM}$. Put $V:=\sM\setminus \cup_{j=1}^k \Sigma_j$ and 
$$\mathscr{B}_{j,r}:= \bar{\B}_\sM(\Sigma_j,r)\cap V,\quad j=1,\cdots,k.$$ 
Let $\d_j={\rm dist}_{\sM}(\cdot, \Sigma_j).$
Furthermore, the singularity function $\rho$ satisfies 
\begin{align}
\label{S5.3: near bdry}
\rho = \d_j^{\beta_j} \quad\text{in } \mathscr{B}_{j,r} 
\end{align}
for some $r\in (0,\delta)$ and $\beta_j\geq 1$, where $\delta< {\rm diam}(\sM)$ and $\mathscr{B}_{i,\delta} \cap \mathscr{B}_{j,\delta}=\emptyset$ for $i\neq j$, and
$$\rho\sim {\bf 1},\quad \text{elsewhere on }V. $$
\item[(ii)] ${\bf W}=\{W_1,\cdots,W_n\}$ is a finite set of disjoint $m$-dimensional wedges.
More precisely, there is a diffeomorphism $f_i: W_i \to W(R_i,B_i,\Gamma_i)$ with $R_i\in \mathscr{C}(J)$.
Let $I_r:=(0,r]$ and
$$\mathscr{G}_{i,r}:=f_i^{-1} (\phi_P(I_r\times B_i)\times\Gamma_i),\quad i=1,\cdots,n.$$
Moreover, the singularity function $\rho$ satisfies 
\begin{align}
\label{S5.3: on wedge end}
\rho = f_i^*(\phi_P^*(R_i|_{I_r} \otimes {\bf 1}_{B_i})\otimes {\bf 1}_{\Gamma_i}) \quad \text{in } \mathscr{G}_{j,r}
\end{align}
for some $r\in (0,1]$, and
$$\rho\sim {\bf 1},\quad \text{elsewhere on } W_i. $$
\item[(iii)] $\{V\}\cup {\bf W}$ forms a covering for $\M$, and $ V\cap W_i \subset \partial V \cap \partial W_i$.
\item[(iv) ] $(\M,g)$ satisfies that $g|_{W_i}= f^*(\phi_{P_i}^*(dt^2 + g_{B_i})+g_{\Gamma_i}), $
where $\phi_{P_i}$ is the canonical stretching diffeomorphism of $P_i(R_i,B_i;J)$.
\end{itemize}
If $(\M,g;\rho)$ satisfies conditions (i)-(iv), then it is called a {\em singular manifold  with  wedge ends and $\boldsymbol{\beta}$-removable singularities}, where ${\boldsymbol{\beta}}=(\beta_1,\cdots,\beta_k)$. 
If in addition, we assume that $(\sM,g)$ is compact, then we call $(\M,g;\rho)$ a {\em compact singular manifold  with  wedge ends and $\boldsymbol{\beta}$-removable singularities}. 
If we further assume that all the $R_i\in \mathscr{C_U}(J)$, then we call  $(\M,g;\rho)$ a {\em (compact)  singular manifold  with uniformly mild wedge ends and  $\boldsymbol{\beta}$-removable singularities}.
We denote $\boldsymbol{\beta}$ by $\boldsymbol{1}$ if $\boldsymbol{\beta}=(1,\cdots,1)$.
\end{definition}
\begin{remark}
\begin{itemize}
\item[]
\item[(i)]  It follows from \cite[Theorem~1.6]{Ama14} that $(\M,g;\rho)$ is a {\em singular manifold}.
\item[(ii) ] The condition $ V\cap W_i \subset \partial V \cap \partial W_i$ in the above definition is not necessary. We impose this condition only for the sake of computational simplicity.
\end{itemize}
\end{remark}

\begin{theorem}
\label{S4: Poincare ineq}
Suppose that $(\M,g;\rho)$ is a $C^1$-compact singular manifold  with wedge ends and $\boldsymbol{\beta}$-removable singularities.
Then every $u\in W^1_2(\M)$ satisfies 
$$ \|u\|_2\leq C\| |\nabla u|_g \|_2 $$
for some fixed $C>0$. 
\end{theorem}
\begin{proof}
Our plan to prove this Poincar\'e inequality is as follows. We will prove our claim first for wedge type manifolds and compact manifolds of type $V$ in part (i) of Definition~\ref{S4: Torn mfd} separately, and then we will glue together the obtained inequalities.

(i) Assume that $R(t)\in \mathscr{C}(J)$. Then by Lemma~\ref{S4: lem 4.1}, $(J\times B, dt^2+ g_B; R\otimes {\bf 1}_B)$ is a {\em singular manifold}. Without loss of generality, we assume that $B$ is connected.
 
Let $B_0:= \{0\}\times B$ and $B_1:=\{1\}\times B.$
Consider the compact manifold $(\bar{J}\times B, dt^2 +g_B)$ with boundary $B_0\cup B_1$, which is of class $C^1$. 
Thus the Rellich-Kondrachov embedding theorem is known to hold on $(\bar{J}\times B, g_C)$, where $g_C=dt^2 +g_B$. More precisely, the embedding 
\begin{align}
\label{S4: cpt emb-1}
W^1_2(\bar{J}\times B)  \hookrightarrow L_2(\bar{J}\times
B)
\end{align}
is compact. Note that by definition, we have
\begin{align}
\label{S4: emb-2}
W^1_2(J\times B) \doteq \zW^1_2(\bar{J}\times B) \hookrightarrow W^1_2(\bar{J}\times B),
\end{align}
where $\zW^1_2(\bar{J}\times B)$ is the closure of $\mathcal{D}(\M)$ in $W^1_2(\bar{J}\times B)$.

It follows from the standard trace theorem that  $\gamma_{B_0}\in \L(W^1_2(\bar{J}\times B), W^{1/2}_2(B_0))$. By the density of $\mathcal{D}(\M)$ in $\zW^1_2(\bar{J}\times B)$ and the fact that 
$$\gamma_{B_0}(u)=0,\quad u\in\mathcal{D}(\M), $$ we immediately infer that
$$\gamma_{B_0}(u)=0,\quad u\in \zW^1_2(\bar{J}\times B).  $$
Therefore, any function $u\in W^1_2(J\times B) $  actually admits a vanishing trace on $B_0$.

(ii) Next we will prove that the Poincar\'e inequality holds for $W^1_2(J\times B)$, i.e., there exists a $C>0$ such that
\begin{align}
\label{S4: Poincare-1}
\|u\|_2\leq C\| |\nabla_{g_C} u|_{g_C}\|_2,\quad u\in  W^1_2(J\times B).
\end{align}
Assume, on the contrary, that there exists a sequence $(u_k)_{k\in\N}$ in $W^1_2(J\times B)$ such that
$$ \|u_k\|_2 \geq k \| |\nabla_{g_C} u_k|_{g_C}\|_2. $$
Let $\displaystyle v_k= \frac{u_k}{\|u_k\|_2}$. Then $\|v_k\|_2=1$ and 
$$\displaystyle  \| |\nabla_{g_C} v_k|_{g_C}\|_2= \frac{\| |\nabla_{g_C} u_k|_{g_C}\|_2}{\|u_k\|_2}\leq \frac{1}{k}.$$
Hence $\|v_k\|_{1,2}= (\|v_k\|_2^2 + \| |\nabla_{g_C} v_k|_{g_C}\|_2^2)^{1/2} \leq 2$. By the compact embedding~\eqref{S4: cpt emb-1} and the embedding \eqref{S4: emb-2}, we can find a $v_\infty\in W^1_2(J\times B)$ such that 
$$v_k\rightharpoonup v_\infty \quad \text{ in } W^1_2(J\times B),\quad \text{and}\quad v_k\to v_\infty \quad \text{ in } L_2(\bar{J}\times B). $$
Therefore, by weak lower semicontinuity, 
$$1\leq \|v_\infty\|_{1,2}  \leq \liminf\limits_{k\to\infty} \|v_k\|_{1,2}=1, $$
which implies
$$ \| |\nabla_{g_C} v_\infty|_{g_C}\|_2=0. $$
Now we can infer that $|\nabla_{g_C} v_\infty|_{g_C}=0$ almost everywhere and thus $v_\infty\equiv c$ on $\bar{J}\times B$. However, by $\gamma_{B_0}(v_\infty)=0$, we have $v_\infty\equiv 0$ on $\bar{J}\times B$, which contradicts $\|v_\infty\|_2=1$.

(iii) Now choose $u\in W^1_2(P)$, where $P=P(R,B;J)$. Let $g_P=\phi_P^* g_C$. Then $\phi_{P,*}u\in W^1_2(J\times B)$, where $\phi_{P,*}=(\phi_P^{-1})^*$, and by \eqref{S4: Poincare-1}
\begin{align*}
\|u\|_{L_2(P)}=\|\phi_{P,*} u\|_{L_2(J\times B)}\leq C\|  |\nabla_{g_C}(\phi_{P,*} u)|_{g_C}\|_{L_2(J\times B)}=C\|  |\nabla_{g_P} u|_{g_P}\|_{L_2(P)}.
\end{align*}

(iv) 
For a model $\Gamma$-wedge over $P(R,B)$, $W(R,B,\Gamma)$, we define a diffeomorphism
$$\phi_P^\Gamma: W(R,B,\Gamma)\to J\times B\times \Gamma:\quad (t,R(t)y,x)\mapsto (t,y,x) .$$
By (i) and (ii), we can prove the Poincar\'e inequality on $J\times B\times \Gamma$. The Poincar\'e inequality on any $W_i$ follows in the same way as in (iii). To sum up, we have proved that
\begin{align}
\label{S4: Poincare-2}
\|u\|_{L_2(W_i)}\leq C\| |\nabla u|_g\|_{L_2(W_i)},\quad u\in  W^1_2(W_i)
\end{align}
for some $C>0$. 

(v)
Now we consider the {\em singular manifold} $V$ in part (i) of Definition~\ref{S4: Torn mfd}. As before, for simplicity, we assume that $V$ is connected. Otherwise, we just consider the problem on each connected component of $V$.

Since $(V,g)$ is a singular manifold, again by definition and Rellich-Kondrachov embedding theorem for compact manifolds, we have
$$W^1_2(V)\doteq \zW^1_2(\sM)\hookrightarrow L_2(\sM), $$
where $\zW^1_2(\sM)$ is the closure of $\mathcal{D}(V)$ in $W^1_2(\sM)$, and the second embedding is compact. 

An analogous argument as in step (i)-(ii) yields
\begin{align}
\label{S4: Poincare-3}
\|u\|_{L_2(V)}\leq C\| |\nabla u|_g\|_{L_2(V)},\quad u\in  W^1_2(V)
\end{align}
for some $C>0$. 

(iv) Our final step is to establish the Poincar\'e inequality on $(\M,g)$ by gluing up together the inequalities obtained in (iv) and (v). Given any $u\in W^1_2(\M)$, then $u|S\in W^1_2(S)$ with $S\in \{V,W_1,\cdots, W_n\}$. It follows from \eqref{S4: Poincare-2} and \eqref{S4: Poincare-3} that
\begin{align*}
\|u\|^2_{L_2(\M)}&= \|u\|^2_{L_2(V)}+\sum_i \|u\|^2_{L_2(W_i)} \leq C (\||\nabla u|_g\|^2_{L_2(V)}+\sum_i \||\nabla u|_g\|^2_{L_2(W_i)}) \\
&=C\| |\nabla u|_g\|^2_{L_2(\M)}.
\end{align*}
This completes the proof.
\end{proof}

\section{\bf Contraction  strongly continuous analytic semigroups}

In this section, we consider a class of second order differential operators of the form
$$
\sA u = -\div (a\, \gd u) 
$$
on $(\M,g;\rho)$, a {\em $C^2$-compact singular manifold  with wedge ends and $\boldsymbol{\beta}$-removable singularities}. Here the coefficient $a$ satisfies the following assumption:
\begin{equation}
\label{S4: ASP coef}
a\in BC^{1,0}(\M,\R_+),\quad \inf\limits_{\p\in\M} a(\p) >0.
\end{equation}
The sesquilinear form associated  with $\sA$ with respect to $L_2(\M)$ is 
$$\a(u,v)=\int_\M ( a \nabla u | \nabla \bar{v})_g\, d\mu_g $$
with $dom(\a)=W^1_2(\M)=:X$.

\begin{prop}
\label{S4: cont-L_2-coer}
$\a$ is continuous and $X$-coercive. More precisely,
\begin{itemize}
\item[](Continuity) there exists some constant $M$ such that for all $u,v\in X$
$$|\a(u,v)|\leq M\|u\|_X \|v\|_X ;$$
\item[]($X$-Coercivity) there is some $M$ such that for any $u\in X$
$$\Rp(\a(u,u)) \geq M \|u\|^2_X.$$
\end{itemize}
\end{prop}
\begin{proof}
(i) A direct computation shows
\begin{align*}
|\a(u,v)|&\leq\int_\M |(\nabla u | \nabla \bar{v})_g|\, d\mu_g \leq \int_\M |\nabla u |_g |\nabla v|_g\, d\mu_g\\
&\leq \||\nabla u |_g\|_2  \||\nabla v |_g\|_2 \leq \|u\|_X \|v\|_X.
\end{align*}
(ii) We have
\begin{align*}
\Rp \a(u,u)=\a(u,u) \geq C \||\nabla u|_g\|_2^2 \geq C^\prime \|u\|_X^2.
\end{align*}
The first inequality follows from \eqref{S4: ASP coef}. Meanwhile, the second inequality is a straightforward conclusion from  Theorem~\ref{S4: Poincare ineq}.
\end{proof}

Based on Proposition~\ref{S4: cont-L_2-coer}, $\a$ with $dom(\a)=X$ is densely defined, sectorial and closed on $L_2(\M)$. By \cite[Theorems~VI.2.1, IX.1.24]{Kato80}, we can associate with $\a$ an operator $T$ such that $-T$ generates a contraction analytic strongly continuous semigroup on $L_2(\M)$, i.e., $\|e^{-tT}\|_{\L(L_2(\M))}\leq 1$ for all $t\geq 0$, with domain 
$$dom(T):=\{u\in X, \exists ! v\in L_2(\M):\a(u,\phi)=\langle v | \phi \rangle, \forall \phi\in X \},\quad T u=v, $$
which is a core of $\a$. 
$T$ is unique in the sense that it is the  only  operator satisfying 
$$\a(u,v)= \langle T u, v \rangle,\quad u\in dom(T),\, v\in X.$$
Looking at the original differential operator $\sA$,  from Proposition~\ref{S2: nabla-unweighted}, \eqref{S2: div-thm-unweighted-1} and \eqref{S2: div-thm-unweighted-2}, we infer that $\sA\in \L(X, X^\prime)$ and
$$\langle \sA u | v \rangle= \a(u,v),\quad u,v\in X.$$
Hence the uniqueness of $T$ implies that
$$\sA|_{dom(T)}=T .$$

\begin{prop}
\label{S4: contraction semigroup-Lp}
$-\sA$ generates a contraction strongly continuous semigroup on $L_p(\M)$  for  $1\leq p< \infty$
and
\begin{equation}
\label{S4: Lap spectrum}
\sup\{ \Rp(\mu): \, \mu\in \sigma(-\sA)\}<0.
\end{equation}
\end{prop}
\begin{proof}
A simple computation shows that $u\in X= W^1_2(M)$ implies that 
$$(|u|-1)^+  \sg u \in W^1_2(M)$$ 
and
$$
\nabla ((|u|-1)^+  \sg u)=
\begin{cases}
\nabla u, \quad & |u|>1;\\
0, & |u|\leq 1.
\end{cases}
$$
Here it is understood that 
\begin{align*}
\sg u:=
\begin{cases}
u/|u|, \quad & u\neq 0;\\
0, &u=0.
\end{cases}
\end{align*}
It is an easy task to show that
$$\Rp(\a(u,(|u|-1)^+  \sg u))\geq 0.$$
A semigroup $T(t)$ is called $L_\infty$-contractive if for all $u\in  L_\infty(\M)\subset L_2(\M)$
$$
\|T(t) u\|_\infty \leq \|u\|_\infty,\quad t\geq 0.
$$
By \cite[Theorem~2.7]{Ouh92}, the semigroup $\{e^{-t\sA}\}_{t\geq 0}$ is $L_\infty$-contractive.

The rest of the proof for the first part of the assertion follows in the same manner as in \cite[Theorem~1.4.1]{Dav89}.

To see that the spectral bound is valid, we look at the operator
$\sA -\omega$ for some sufficiently small
positive $\omega$. 
From Theorem~\ref{S4: Poincare ineq}, we can show by following 
the above argument step by step that the first part of the 
assertion still holds true for $\sA-\omega$ with $\omega$ small.
This  immediately gives a spectral bound for $\sA$.
\end{proof}

\begin{prop}
\label{S4: analytic contraction semigroup-Lp}
$-\sA$ generates a contraction analytic strongly continuous semigroup on $L_p(\M)$ for $1< p <\infty$.
\end{prop}
\begin{proof}
The proof  follows in the same way as that of \cite[Theorem~3.7]{Shao15} with $\lambda=2$ and $\lambda^\prime=0$. 
\end{proof}

We denote the $L_p$-realization of $\sA$ by $\sA_p$. 
Let $\R_\M$ be the set of all real-valued component-wise constant functions on $\M$. 
\begin{lem}
Assume that $(\M,g;\rho)$ is a  $C^2$-compact singular manifold  with uniformly mild wedge ends and $\boldsymbol{1}$-removable singularities. Suppose that the operator $\sA_p u=-\div(a \gd u)$ and $a=C_M + \hat{a}$, where $C_\M \in\R_\M$ satisfies $C_\M>0$ and $\hat{a}\in BC^{1,\vartheta}(\M)$ with $\vartheta<0$ and $\|\hat{a}\|_\infty<\inf C_M$.
Then
$$dom(\sA_p)\doteq W^{2,-2}_p(\M)\quad \text{for all }1<p<\infty.$$
\end{lem}
\begin{proof}
It is clear that the operator $-C_\M \Delta_p$ fulfils the condition~\eqref{S4: ASP coef} and thus all the statements established for $\sA$ in this section holds true for $-C_\M \Delta_p$.

On the other hand, by a slight modification of the proof for \cite[Corollary~5.20]{Shao15}, we infer that for $\omega>0$ sufficiently large
$$C_\M \Delta_p+ \omega  \in \Lis (W^{2,-2}_p(\M), L_p(\M)).$$
This immediately implies that $dom(C_\M \Delta_p)\doteq W^{2,-2}_p(\M)$. 

Then the assertion follows from a similar perturbation argument to the proof of \cite[Theorem~6.1]{RoiSch15}. For the reader's convenience, we will briefly state the idea. 
We consider the manifold $(\M_r,g;\rho)$ with boundary, where $\M_r:=\cup_j \mathscr{B}_{j,r} \bigcup \cup_j \mathscr{G}_{i,r}$. Then by the same argument as above, one can show that 
$C_\M \Delta_p$ generates a contraction analytic strongly continuous semigroup on $L_p(\M_r)$ with domain $\mathring{W}^{2,-2}_p(\M_r)$, where $\mathring{W}^{2,-2}_p(\M_r)$ denotes the subspace of $W^{2,-2}_p(\M_r)$ with vanishing trace on $\partial\M_r$ whenever it exists. 
By the resolvent formula and the contractivity of $-C_\M \Delta_p$, 
$$
(\lambda+C_\M\Delta_p)^{-1}=\int_0^\infty e^{-\lambda t} e^{tC_\M \Delta_p}\, dt, \quad \lambda>0,
$$
we can show that $\|\Delta_p(\lambda+\Delta_p)^{-1}\|_{\L(L_p(\M_r)}$ is uniformly bounded for $r$ small.
For any $\varepsilon>0$, there is a sufficiently small $r=r(\varepsilon)>0$ such that 
$$
\|-C_M\Delta_p -\sA_p\|_{\L(\mathring{W}^{2,-2}_p(\M_r), L_2(\M_r))} \leq \varepsilon.
$$
Then by the standard perturbation theory for linear semigroups (cf. \cite[Theorem~1.5]{DenHiePru03}), we conclude that
$\sA_p$ generates an analytic strongly continuous semigroup on $L_p(\M_r)$ with domain $\mathring{W}^{2,-2}_p(\M_r)$ as long as $r$ is small enough.
The manifolds $(\M_0,g):=(\M\setminus \M_{r/2},g)$ is a {\em uniformly regular Riemannian manifold}, which can be viewed as a {\em singular manifold} with $\rho\sim {\bf 1}_\M$. By \cite[Theorem~5.2]{Ama13b}, $\sA_p$ 
generates an analytic strongly continuous semigroup on $L_p(\M_0)$ with domain $\mathring{W}^{2,-2}_p(\M_0)$. The last step is to glue together the results on $\M_r$ and $\M_0$ and conclude that, for $\omega$ sufficiently large, 
$$
\sA_p+\omega \in \Lis (W^{2,-2}_p(\M), L_p(\M)).
$$
This implies the desired result. This idea was exhibited in the proof of \cite[Theorem~5.17]{Shao15}.

A different proof for this fact can also be found in \cite[Section~3]{Shao1603}.
\end{proof} 
\begin{remark}
Note that the differential operator $\sA_p$ in the above lemma satisfies \eqref{S4: ASP coef} and thus the spectrum bound~\eqref{S4: Lap spectrum} still holds true for $\sA_p$.
\end{remark}

We consider the following abstract Cauchy problem 
\begin{equation}
\label{S4: Cauchy problem}
\left\{\begin{aligned}
\partial_t u(t) +\cA u(t) &=f(t), &&t\geq 0\\
u(0)&=0 . &&
\end{aligned}\right. 
\end{equation}

For $\theta\in (0,\pi]$, the open sector with angle $2\theta$ is denoted by
$$\Sigma_\theta:= \{\omega\in \mathbb{C}\setminus \{0\}: |\arg \omega|<\theta \}. $$
\begin{definition}
Let $X$ be a complex Banach space, and $\cA$ be a densely defined closed linear operator in $X$ with dense range. $\cA$ is called sectorial if $\Sigma_\theta \subset \rho(-\cA)$ for some $\theta>0$ and
$$ \sup\{|\mu(\mu+\cA)^{-1}| : \mu\in \Sigma_\theta \}<\infty. $$
The class of sectorial operators in $X$ is denoted by $\S(X)$. 
\end{definition}

\begin{definition}
Assume that $X_1\overset{d}{\hookrightarrow}X_0$ is some densely embedded Banach couple.
Suppose that $\cA\in \S(X_0)$ with $dom(\cA)=X_1$.
The Cauchy problem \eqref{S4: Cauchy problem} has maximal $L_p$-regularity if for any 
$$f\in L_p([0,\infty), X_0) ,$$
\eqref{S4: Cauchy problem} has a unique solution
$$u\in L_p([0,\infty), X_1) \cap H^1_p([0,\infty), X_0) .$$
We denote this by 
$$\cA\in \mathcal{MR}_p(X_1, X_0).$$
\end{definition}
We will refer the reader to \cite{DenHiePru03, Pru03, PruSim16} for more details of maximal regularity theory.

Following the proof of \cite[Theorem~4.8]{Shao1502}, we can prove the following maximal regularity result.
\begin{theorem}
\label{S4: Thm-MR}
Suppose that $(\M,g;\rho)$ is a $C^2$-compact singular manifold with uniformly mild wedge ends and $\boldsymbol{1}$-removable singularities. Assume that the differential operator
$$
\sA u = -\div (a\, \gd u) 
$$ 
and $a=C_M + \hat{a}$, where $C_\M \in\R_\M$ satisfies $C_\M>0$ and $\hat{a}\in BC^{1,\vartheta}(\M)$ with $\vartheta<0$ and $\|\hat{a}\|_\infty<\inf C_M$. 
Then for all $1<p<\infty$
$$ 
\sA \in \mathcal{MR}_p(W^{2,-2}_p(\M), L_p(\M)).
$$
\end{theorem}


\section{\bf $L_p$-Stability of the porous medium equations}

Suppose that $(\M,g;\rho)$ is an $m$-dimensional   {\em $C^2$-compact singular manifolds  with uniformly mild wedge ends and $\boldsymbol{1}$-removable singularities}, which might not be connected.

We will establish existence and uniqueness of solutions to the following porous medium equation first. This part is comparable to the result in \cite{RoiSch15}.
\begin{equation}
\label{S5: PME-1}
\left\{\begin{aligned}
\partial_t u -\Delta (|u|^{n-1}u  )&=0   &&\text{on}&&\M\times (0,\infty);\\
u(0)&=u_0   &&\text{on}&&\M
\end{aligned}\right.
\end{equation}
for $n> 0$ and thus includes the fast diffusion equation. Here $\Delta=\Delta_g$, the Laplace-Beltrami operator with respect to $g$.
Since $\Delta (|u|^{n-1}u  )= n\div(|u|^{n-1}\gd u)$, by a rescaling of the temporal variable, \eqref{S5: PME-1} is equivalent to 
\begin{equation}
\label{S5: PME}
\left\{\begin{aligned}
\partial_t u -\div(|u|^{n-1}\gd u) &=0   &&\text{on}&&\M\times (0,\infty);\\
u(0)&=u_0   &&\text{on}&&\M.
\end{aligned}\right.
\end{equation}
This formulation also includes the logarithmic diffusion when $n=0$.
A direct computation shows that, by setting $f=|u|^{n-1}u$, \eqref{S5: PME} is equivalent to
\begin{equation*}
\left\{\begin{aligned}
\partial_t f - |f|^{\frac{n-1}{n}}\Delta f &=0   &&\text{on}&&\M\times (0,\infty);\\
f(0)&=f_0   &&\text{on}&&\M.
\end{aligned}\right.
\end{equation*} 
However, as we will see later,  the divergence form of \eqref{S5: PME} will assist us in obtaining the stability of solutions to the porous medium equation. For this reason, we will stay with the formulation in \eqref{S5: PME} in this section.

We assume that $p>m+2$ and set
$$E_0:=L_p(\M, \R),\quad E_1:=W^{2,-2}_p(\M, \R), $$
and
$$E_{1/p}:= W^{2-2/p,\frac{2}{p}-2}_p(\M, \R).$$ 
By Property~\ref{S2: embedding}, we have
\begin{equation}
\label{S5.1: crucial embedding}
E_{1/p}\hookrightarrow BC^{1,\frac{m+2}{p}-2}(\M, \R)\quad \text{and }\frac{m+2}{p}-2<0.
\end{equation}
We put 
$$B_R:=\{u\in BC^{1,\frac{m+2}{p}-2}(\M, \R): \, \|u\|_\infty<R\},$$ 
which is open in $BC^{1,-2+\frac{m+2}{p} }(\M, \R)$ by the embedding $BC^{1,-2+\frac{m+2}{p}}(\M, \R)\hookrightarrow BC^{0,0}(\M, \R)$,
and let
$$U_R:=\iota^{-1}(B_R),$$
where $\iota$ is the embedding map of $E_{1/p}\hookrightarrow BC^{1,\frac{m+2}{p}-2}(\M, \R)$.

 
Suppose that the initial datum is of the form $u_0= C_\M+ w_0$ with $C_\M \in\R_\M$ satisfying $|C_\M|>0$ and $w_0\in E_{1/p}$. 
Assumet that $w_0\in U_R$ with $R<\frak{b}:=\inf |C_\M|$.

We first consider the linearized porous medium equation at $u=C_\M$.
\begin{equation}
\label{S5: LPME}
\left\{\begin{aligned}
\partial_t u -|C_\M|^{n-1}\Delta u &=0   &&\text{on}&&\M\times (0,\infty);\\
u(0)&=C_\M   &&\text{on}&&\M.
\end{aligned}\right.
\end{equation}
Clearly, $u^*=C_\M$ is a solution to \eqref{S5: LPME}.

Next we consider the nonlinear part.
\begin{equation}
\label{S5: NPME}
\left\{\begin{aligned}
\partial_t u -\div( |u^*+u|^{n-1}\gd u) &=0   &&\text{on}&&\M\times (0,\infty);\\
u(0)&=w_0   &&\text{on}&&\M.
\end{aligned}\right.
\end{equation}
If \eqref{S5: NPME} admits a solution $\bar{u}$, then $\hat{u}=u^*+\bar{u}$ solves \eqref{S5: PME}.

\begin{lem}
\label{S51: lem1}
For any $u\in U_R$ with $R<\frak{b}$, $|u^*+u|^{n-1}\in BC^{1,\vartheta}(\M, \R)\oplus \R_\M$ with some $\vartheta<0$ and
$$ [u\mapsto |u^*+u|^{n-1}]\in C^\omega(U_R, BC^{1,0}(\M, \R)).$$
Here $\omega$ is the symbol of real analyticity. Moreover, we can find some $C>0$ such that
$$
|u^*+u|^{n-1}>C,\quad u\in U_R.
$$
\end{lem}
\begin{proof}
By the embedding~\eqref{S5.1: crucial embedding} and the choice of $R$, it is clear that $|u^* +u| \in BC^{1,0}(\M, \R)$ and we can find some $c\geq 1$ such that 
\begin{equation}
\label{S5: lower bd}
1/c<|u^* +u| <c ,\quad u\in U_R. 
\end{equation}
Now it follows from Property~\ref{S2: substituion op} that
$$[u\mapsto |u^* +u|^{n-1}] \in C^\omega(U_R, BC^{1,0}(\M,\R)).$$
Note that $v:=u/u^* \in BC^{1,\frac{m+2}{p}-2}(\M,\R)$ satisfies $\|v\|_\infty<1$.
So it is clear that 
$$ 
\inf|u^*+u|^{n-1}=|u^*|^{n-1}\inf |1+v|^{n-1}>0.
$$ 
It only remains to prove that $|u+u^*|^{n-1}-|u^*|^{n-1}\in BC^{1,\vartheta}(\M,\R) $ for some $\vartheta<0$.

By the Taylor  expansion, for any $\p\in \M$
$$(1+v(\p))^{n-1}-1=  \binom{n-1}{1} v(\p) +  (n-1)(n-2)(\frac{1-\theta(\p)}{1+\theta(\p)v(\p)})(1+\theta(\p)v(\p))^{n-2} v^2(\p), $$
where $\theta(\p)\in (0,1)$. 
Note that $\frac{m+2}{p}-2<-1$.
We have 
$$ \rho^{-1}(\p)[\binom{n-1}{1} v(\p) +  (n-1)(n-2)(\frac{1-\theta(\p)}{1+\theta(\p)v(\p)})(1+\theta(\p)v(\p))^{n-2} v^2(\p)]<M $$
for some $M>0$, and similarly
$$\| \nabla_g[|{\bf 1}_\M+v|^{n-1}-{\bf 1}_\M] \|_\infty<\infty.$$
\end{proof}

By Property~\ref{S2: nabla}, Lemma~\ref{S51: lem1}, \eqref{S5: lower bd} and Theorem~\ref{S4: Thm-MR}, for $R<\frak{b}$, 
$$[u\mapsto -\div( |u^*+u|^{n-1} \gd \cdot)]\in C^\omega(U_R, \mathcal{MR}_p(E_1,E_0)),\quad u\in U_R.$$
\cite[Theorem~2.1]{CleLi93}  now implies that there exists a unique solution to \eqref{S5: NPME} 
$$\bar{u}\in \bE(J):=L_p(J,E_1)\cap H^1_p(J, E_0) $$
for some $J=[0,T]$.

Next we will prove that $\hat{u}=u^*+\bar{u}$ is actually the unique solution to \eqref{S5: PME} in the class $\bE(J)\oplus \R_\M$.
By \eqref{S5.1: crucial embedding} and the well-known embedding theorem, see \cite[formula~(2.1)]{CleLi93} for example, we have
$$\bE(J) \hookrightarrow C(J, E_{1/p})\hookrightarrow C(J, BC^{1,\frac{m+2}{p}-2}(\M, \R)).$$
Because of $\frac{m+2}{p}-2<0$, we infer that $BC^{1,\frac{m+2}{p}-2}(\M, \R)\cap \R_\M={\bf 0}_\M$, which implies $\bE(J)\cap \R_\M={\bf 0}_\M$. 

Assume that there is a solution $v\in \bE(J)\oplus \R_\M$ to \eqref{S5: PME}, then it has a unique decomposition $v=v_1+v_2$ with $v_1\in \bE(J)$ and $v_2\in \R_\M$.  Plugging in the initial condition yields $v(0)=v_1(0)+v_2=w_0+C_\M$. So $v_2\equiv C_\M$. By the uniqueness of solution to \eqref{S5: NPME}, $\hat{u}=v$.

Next we look at the equilibria of \eqref{S5: PME}. It is clear that the stationary solutions to \eqref{S5: PME} are in $\R_\M$. 
For any $u\in E_{1/p}$, we put $F(u)=\div(|u^*+u|^{n-1}\gd u)$. Then an easy computation shows that 
$$\partial F(0)=|u^*|^{n-1}\Delta ,$$
where $\partial F(0)$ is the Fr\'echet derivative of $F$ at $0$. 
It has been shown in \eqref{S4: Lap spectrum} that 
\begin{equation}
\label{S5.2: spectrum}
\sup\{ \Rp(\mu): \, \mu\in \sigma(\partial F(0))\}<0.
\end{equation}

Now the asymptotic stability of the equilibria is an immediate consequence from  the well-known linearized stability theorem. We are ready to state the wellposedness and stability theorem for \eqref{S5: PME}.

\begin{theorem}
Suppose that $(\M,g;\rho)$ is an $m$-dimensional  $C^2$-compact singular manifolds  with uniformly mild wedge ends and $\boldsymbol{1}$-removable singularities, which might not be connected.  Let $\R_\M$ be the set of all real-valued component-wise constant functions on $\M$. Assume that $p>m+2$. 
\begin{itemize}
\item[(i)] Then for any initial value $u_0=C_\M + w_0$, where $C_\M\in \R_\M$ satisfies $|C_\M|>0$ and $w_0\in W^{2-2/p,\frac{2}{p}-2}_p(\M, \R)$ satisfies $\|w_0\|_\infty<\inf |C_\M|$, the porous medium equation~\eqref{S5: PME} has a unique solution  
$$u\in L_p(J,W^{2,-2}_p(\M, \R) ) \cap 	H^1_p(J, L_p(\M, \R)) \oplus \R_\M$$
for some $J=[0,T] $ with $T=T(u_0)>0$. Moreover,
$$u\in C(J, W^{2-2/p,\frac{2}{p}-2}_p(\M, \R))\oplus \R_\M. $$
\item[(ii)] Any  $C_\M\in \R_\M$ with $|C_\M|>0$ is a stationary solution to \eqref{S5: PME} and attracts all solutions which  are initially $W^{2-2/p,\frac{2}{p}-2}_p(\M, \R)$ close to $C_\M$. More precisely, if the initial datum $u_0=C_\M + w_0$ satisfies $w_0\in W^{2-2/p,\frac{2}{p}-2}_p(\M, \R)$ and $\|w_0\|_{2-2/p,p;\frac{2}{p}-2}$ sufficiently small, then the solution $u$ to \eqref{S5: PME} converges to $C_\M$ exponentially fast in $W^{2-2/p,\frac{2}{p}-2}_p(\M, \R)$-topology, in particular, in $C^1$-topology.
\end{itemize}
\end{theorem}
\begin{remark}
\begin{itemize}
\item[]
\item[(i) ] The wellposedness part can be easily adapted to inhomogeneous porous medium equations with a forcing term $f\in L_p(J,L_p(\M, \R)) $ to obtain the existence and uniqueness of solutions.
\item[(ii) ] If we only focus on the local well-posedness of \eqref{S5: PME}, we can actually relax the restriction on the initial data a little bit. By using the results in \cite[Section~3]{Shao1603}, we can actually show the following existence and uniqueness result. Suppose that $(\M,g;\rho)$ is an $m$-dimensional {\em $C^2$-singular manifold} satisfying the condition (A2') in \cite[Section~3]{Shao1603}. 
If $p>m+2$ and $\vartheta<-m/p$, then for any initial value $u_0=C_\M + w_0$, with  $w_0\in W^{2-2/p,\vartheta}_p(\M, \R)$,  \eqref{S5: PME} has a unique solution  
$$u\in L_p(J,W^{2,\vartheta-2/p}_p(\M, \R) ) \cap 	H^1_p(J, L_p^{\vartheta+2-2/p}(\M, \R)) \oplus \R_\M.$$
Similar result also holds true for inhomogeneous porous medium equation. 
\end{itemize}
\end{remark}

\section{\bf Global mild solutions to the porous medium equations}

In this section, we will use the contraction semigroup result established 
in Section~4 and the nonlinear semigroup theory (cf. \cite{Bar73}) to establish the global 
existence for $L_1$-mild solutions to the porous medium equation
\begin{equation}
\label{S6: PME}
\left\{\begin{aligned}
\partial_t u -\Delta (|u|^{n-1} u) &=0   &&\text{on}&&\M\times (0,\infty);\\
u(0)&=u_0   &&\text{on}&&\M
\end{aligned}\right.
\end{equation}
for $n>1$, i.e.,  the slow diffusion case, on an $m$-dimensional $C^2$-compact singular manifold  $(\M,g;\rho)$  with  wedge ends and $\boldsymbol{\beta}$-removable singularities.
In general, nonlinear semigroup theory applies to multi-valued operators. 
But in order to simply our argument, we will only focus our attention on 
single-valued operator, which is enough to deal with the porous medium 
equation.


Let $X_\R$ be a real Banach lattice with an order $\leq$. 
See \cite[Chapter~C-I]{ArenGrohNage86}. 
The complexification of $X_\R$ is a complex Banach lattice defined by
\begin{equation}
\label{S4.1: Banach lattice}
X:=X_\R \oplus i X_\R. 
\end{equation}
The positive cone of $X_\R$ is defined as
$$X_\R^+:=\{x\in X_\R:\, 0\leq x\}. $$
\begin{definition}
Let $X$ be a complex Banach lattice defined as in \eqref{S4.1: Banach lattice}.
Suppose that $\cA\in \S(X)$. Then the semigroup $\{e^{-t\cA}\}_{t\geq 0}$ is real if 
$$e^{-t\cA}X_\R \subset X_\R ,\quad t\geq 0.$$ 
We say that $\{e^{-t\cA}\}_{t\geq 0}$ is positive if 
$$e^{-t\cA}X_\R^+ \subset X_\R^+  ,\quad t\geq 0.$$ 
\end{definition}

\begin{lem}
\label{S6: Positivity}
For all $\lambda>0$ and $u\in L_1(\M)$,
$$ 
\sup (\id -\lambda \Delta)^{-1}u \leq \max \{0, \sup u\}.
$$
\end{lem}
\begin{proof}
By semigroup theory and Proposition~\ref{S4: contraction semigroup-Lp}, for any $\lambda>0$ 
\begin{equation}
\label{S6: eq1}
(\lambda - \Delta)^{-1} u= \int_0^\infty e^{-\lambda t} e^{t\Delta}u\, dt.
\end{equation}
It follows from the $L_\infty$-contractivity of $\Delta$ shown in the proof for Proposition~\ref{S4: contraction semigroup-Lp} that
$$
\| (\lambda - \Delta)^{-1} u\|_\infty \leq \int_0^\infty e^{-\lambda t} \|u\|_\infty\,dt = \frac{1}{\lambda} \|u\|_\infty.
$$
This implies that for all $\lambda>0$ and $u\in L_\infty(\M)$
$$
\| (\id -\lambda \Delta)^{-1}u \|_\infty \leq \|u\|_\infty.
$$
If $u\in L_1(\M)\setminus L_\infty(\M)$, then the above inequality apparently holds true. Thus we actually proved that
\begin{equation}
\label{S6: eq2}
\| (\id -\lambda \Delta)^{-1}u \|_\infty \leq \|u\|_\infty,\quad \lambda>0,\, u\in L_1(\M).
\end{equation} 

Next we will show that the semigroup $\{e^{t\Delta}\}_{t\geq 0}$ is positive on $L_2(\M)$. 
It is easy to see that 
\begin{center}
$u\in W^{2,-2}_2(\M)$ implies that $\bar{u}\in W^{2,-2}_2(\M)$ and $\Delta \bar{u}=\overline{\Delta u}$. 
\end{center}
Here $\bar{u}$ stands for the complex conjugate of $u$. 
By \cite[Chapter~C-II Remark~3.1]{ArenGrohNage86}, the semigroup $\{e^{t\Delta}\}_{t\geq 0}$ is real.

On the other hand, for any $u\in W^1_2(\M,\R)$ implies that $u^+\in  W^1_2(\M)$ and
$$ 
\a(u^+, u^-)=\int_\M (\nabla u^+ | \nabla u^-)_g\, d\mu_g=0. 
$$
By \cite[Theorem~2.4]{Ouh92}, we conclude that $\{e^{t\Delta}\}_{t\geq 0}$ is positive on $L_2(\M)$. 
Furthermore, pick $u\in L_1(\M,\R_+)$ and a sequence $(u_k)_k$ in $L_2(\M)$ converging to $u$ in $L_1(\M)$. Without loss of generality, we may assume $u_k\in L_2(\M,\R_+)$.
If 
$$
\| (e^{t\Delta}u)^-\|_1\geq C>0,
$$
then by the positivity of $\{e^{t\Delta}\}_{t\geq 0}$ on $L_2(\M)$, we have
$$
\| e^{t\Delta}(u-u_k)\|_1 \geq \| (e^{t\Delta}u)^-\|_1 \geq C>0.
$$ 
A contradiction. This implies that $\{e^{t\Delta}\}_{t\geq 0}$ is positive on $L_1(\M)$.
For any $u\in L_1(\M,\R_+)$, applying \eqref{S6: eq1} once more yields
$$
(\lambda-\Delta)^{-1} u =\int_0^\infty e^{-\lambda t} e^{t\Delta }u\, dt \geq 0,\quad \lambda>0.
$$
This gives the positivity of $(\id -  \lambda\Delta)^{-1}$ on $L_1(\M)$, i.e., for all $\lambda>0$
$$
(\id-  \lambda\Delta)^{-1} L_1(\M,\R_+)\subset L_1(\M,\R_+).
$$

Now given any $u\in L_1(\M)$, we decompose it into $u=u^+ - u^-$. 
By \eqref{S6: eq2}, it holds
$$
\sup(\id - \lambda\Delta)^{-1} u^+ \leq \|(\id - \lambda\Delta)^{-1} u^+\|_\infty \leq \|u^+\|_\infty= \sup u^+.
$$
On the other hand, since $-u^-\leq 0$, it follows from the positivity of $(\id - \lambda\Delta)^{-1} $ on $L_1(\M)$ that
$$
\sup (\id - \lambda\Delta)^{-1} (-u^-)\leq 0.
$$
Combining together these two estimates, the asserted claim then follows.
\end{proof}

In the sequel, we denote the $L_1(\M)$-realization of $\Delta$ by $\Delta_1$.
\begin{lem}
\label{S6: inverse}
There exists some $\alpha>0$ such that
$$
\alpha\|u\|_1 \leq \|\Delta_1 u\|_1, \quad u\in dom(\Delta_1).
$$
\end{lem}
\begin{proof}
The assertion follows from Proposition~\ref{S4: contraction semigroup-Lp}.
\end{proof}

Let $\Phi(x)=|x|^{n-1}x$ and $\beta=\Phi^{-1}$. Then they are maximal monotone graphs in $\R^2$ containing $(0,0)$. Moreover, for any $\lambda>0$, so is $\beta/\lambda$.

By Proposition~\ref{S4: contraction semigroup-Lp}, Lemmas~\ref{S6: Positivity} and \ref{S6: inverse}, we can apply \cite[Theorem~1]{BreStr73} and prove the following theorem.
\begin{theorem}
\label{S6: semilinear thm}
For any $f\in L_1(\M)$ and all $\lambda>0$, there exists a unique solution $u\in dom(\Delta_1)$ to 
$$ -\Delta u +\beta(u)/\lambda=f .$$
Moreover, for any $f_1,f_2\in L_1(\M)$, the corresponding solutions $u_1,u_2$ satisfy  
$$
\|\beta(u_1)/\lambda - \beta(u_2)/\lambda \|_1\leq \|f_1-f_2\|_1.
$$
\end{theorem}

\vspace{1em}

\begin{definition}
\cite[Chapter II.3]{Bar73}
\begin{itemize}
\item[]
\item[(i)] A nonlinear operator $\cA$ defined in a Banach space $X$ is called accretive if for all $\lambda>0$
\begin{equation}
\label{S6: contraction}
\|(\id +\lambda \cA) x_1 - (\id +\lambda \cA) x_2\|_X\geq \|x_1-x_2\|_X,\quad x_1,x_2\in dom(\cA).
\end{equation}
\item[(ii)] A nonlinear operator $\cA$ defined in a Banach space $X$ is called $m$-accretive if $\cA$ is accretive and it satisfies the range condition
$$
Rng(\id +\lambda \cA)=X, \quad \lambda>0.
$$
\end{itemize}
\end{definition}

Let $\cA(u)=-\Delta\Phi(u)$ and $dom(\cA)=\beta(dom(\Delta_1))$. Then by the strict monotonicity of $\Phi$, it is a bijection from $dom(\cA)$ to $dom(\Delta_1)$. Note that $u\in dom(\cA)$ iff $\Phi(u)\in dom(\Delta_1)\subset L_1(\M)$. This shows that $u\in L_n(\M)\subset L_1(\M)$.

Theorem~\ref{S6: semilinear thm} implies that, for any $\lambda>0$, $(\id+\lambda\cA)^{-1}$ is a bijection between $dom(\cA)$ and $L_1(\M)$, and \eqref{S6: contraction} is fulfilled. Therefore, $\cA(u)=-\Delta\Phi(u)$ is $m$-accretive. 

To characterize the domain of $\cA$, we will first prove that  
$$dom(\cA)=\{u\in L_1(\M):\, \cA(u)\in L_1(\M)\}. $$
Indeed, as we have seen that
$$
dom(\cA)=\beta(dom(\Delta_1))\subset L_1(\M).
$$
As $v\in L_1(\M)$ belongs to $dom(\Delta_1)$ iff $\Delta_1 v\in L_1(\M)$, the above inclusion reveals that $u\in L_1(\M)$ belongs to $dom(\cA)$ iff $\cA(u)\in L_1(\M)$. 

For any $u\in W^{2,-2}_p(\M)$ with $p>sn$, we have for sufficiently small $s>1$
$$
\int_\M |\rho^{-2n}|u|^n|^s\,d\mu_g= \int_\M |\rho^{-2} u|^{sn}\,d\mu_g \leq C \| u \|_{p;-2}. 
$$
An easy computation shows that $\nabla \Phi(u)= n\nabla u |u|^{n-1}$. It holds that
\begin{align*}
\int_\M |\rho^{-2n+1}|\nabla u|_g |u|^{n-1}|^s\, d\mu_g &= \int_\M (\rho^{-1}|\nabla u|_g )^s (|\rho^{-2}u|)^{s(n-1)}\, d\mu_g\\
&\leq C \|\nabla u\|_{p;-2}^s \|u\|_{p;-2}^{p-s}.
\end{align*}
Since $\nabla^2 \Phi(u)= n \nabla^2 u |u|^{n-1} + n(n-1)(\sg u)\nabla u^{\otimes 2} |u|^{n-2}$, by a similar estimate as for  $\nabla \Phi(u)$, we have
\begin{align*}
\|\nabla^2 \Phi(u)\|_{s;-2n}\leq C(\|\nabla^2 u\|_{p;-2}^s \|u\|_{p;-2}^{p-s} + \|\nabla u\|_{p;-2}^{2s}\|u\|_{p;-2}^{p-2s}).
\end{align*}
In sum, the above estimates show that $\Phi(W^{2,-2}_p(\M))\subset W^{2,-2n}_s(\M)$ for some $s>1$ and $p>sn$. Applying Properties~\ref{S2: nabla} and \ref{S2: embedding} yields
$$
\cA(u)\in L^{2-2n}_s(\M) \hookrightarrow L_1(\M),\quad u\in W^{2,-2}_p(\M),
$$
which further implies that
\begin{center}
$W^{2,-2}_p(\M)\subset dom(\cA) \Longrightarrow dom(\cA)$ is dense in $L_1(\M)$.
\end{center}

Now we can apply the well-known Crandall-Liggett generation theorem \cite[Theorem~I]{CraLig71} to prove the $L_1$-global mild solution to \eqref{S6: PME}. 
Mild solutions are defined as the limit of a sequence of approximation solutions by implicit time discretization. 
See the introduction or \cite[Chapter 10.2]{Vaz07} for more details.
\begin{theorem}
\label{S6: global exist thm}
Suppose that $(\M,g;\rho)$ is an $m$-dimensional $C^2$-compact singular manifold  with  wedge ends and $\boldsymbol{\beta}$-removable singularities. 
Assume that $n>1$. 
For every $u_0\in L_1(\M)$, \eqref{S6: PME} has a unique global solution 
$$
u\in C([0,\infty), L_1(\M)).
$$	
Moreover, every two solutions $u_1$  and $u_2$ corresponding to the initial data $u_{0,1}$ and $u_{0,2}$ satisfy
$$
\|u_1 -u_2\|_1 \leq \|u_{0,1} -u_{0,2} \|_1.
$$
\end{theorem}


\section*{Acknowledgements}
The author would like to express his gratitude to Prof. Elmar Schrohe for helpful discussions.

\end{document}